\documentclass[10pt]{amsart}
\usepackage{latexsym,amssymb,amsmath,epsfig}
\newtheorem{lemma}{Lemma}[section]
\def\proof{{\it Proof.}\nobreak\\}%% proof
\def\qed{\hfill$\Box$ \bigskip}%%%%% end of proof
\newtheorem{theorem}{Theorem}[section]
\newtheorem{question}{Question}[section]
\newtheorem{proposition}{Proposition}[section]

\newtheorem{problem}{Open problem}[section]

\begin{document}

\title[Distance labelings: a generalization of Langford sequences]{Distance labelings: a generalization of Langford sequences}

\author{S. C. L\'opez}
\address{%
Departament de Matem\`{a}tica Aplicada IV\\
Universitat Polit\`{e}cnica de Catalunya. BarcelonaTech\\
C/Esteve Terrades 5\\
08860 Castelldefels, Spain}
\email{susana@ma4.upc.edu}

\author{F. A. Muntaner-Batle}
\address{Graph Theory and Applications Research Group \\
 School of Electrical Engineering and Computer Science\\
Faculty of Engineering and Built Environment\\
The University of Newcastle\\
NSW 2308
Australia}
\email{famb1es@yahoo.es}
\date{\today}
\maketitle

\begin{abstract}
A Langford sequence  of order $m$ and defect $d$ can be identified with a labeling of the vertices of a path of order $2m$ in which each labeled from $d$ up to $d+m-1$ appears twice and in which the vertices that have been label with $k$ are at distance $k$. In this paper, we introduce two generalizations of this labeling that are related to distances.


\end{abstract}

\maketitle
\section{Introduction}
For the graph terminology not introduced in this paper we refer the reader to \cite{Wa,W}. For $m\le n$, we denote the set $\{m,m+1,\ldots, n\}$ by $[m,n]$.
A {\it Skolem sequence} \cite{NicMar67,Sko57} of order $m$ is a sequence of $2m$ numbers $(s_1,s_2,\ldots,s_{2m})$ such that (i) for every $k\in [1,m]$ there exist exactly two subscripts $i,j\in [1,2m]$ with $s_i=s_j=k$, (ii) the subscripts $i$ and $j$ satisfy the condition $|i-j|=k$.
The sequence $(4,2,3,2,4,3,1,1)$ is an example of a Skolem sequence of order $4$.
It is well known that Skolem sequences of order $m$ exist if and only if $m\equiv 0$ or $1$ (mod $4$).

Skolem introduced in \cite{Sko58} what is now called a {\it hooked Skolem sequence} of order $m$, where there exists a zero at the second to last position of the sequence containing $2m+1$ elements. Later on, in 1981, Abrham and Kotzig \cite{AbrKot81}
introduced the concept of \textit{extended Skolem sequence}, where the zero is allowed to appear in any position of the sequence.
An extended Skolem sequence of order $m$ exists for every $m$. The following construction was given in \cite{AbrKot81}:
\begin{equation}\label{eq: extended Skolem}
(p_m,p_m-2,p_m-4,\ldots,2,0,2,\ldots,p_m-2,p_m,q_m,q_m-2,q_m-4,\ldots,3,1,1,3,\ldots,q_m-2,q_m),
\end{equation}
where $p_m$ and $q_m$ are the largest even and odd numbers not exceeding $m$, respectively.
 Notice that from every Skolem sequence we can obtain two trivial extended Skolem sequences just by adding a zero either in the first or in the last position.
 %In this paper, all extended Skolem sequences that we refer to are non trivial, unless otherwise specified.

Let $d$ be a positive integer. A {\it Langford sequence} of order $m$ and defect $d$ \cite{Sim} is a sequence $(l_1,l_2,\ldots, l_{2m})$ of $2m$ numbers such that (i) for every $k\in [d,d+m-1]$ there exist exactly two subscripts $i,j\in [1,2m]$ with $l_i=l_j=k$, (ii) the subscripts $i$ and $j$ satisfy the condition $|i-j|=k$. Langford sequences, for $d=2$, where introduced in \cite{Langford} and they are referred to as \textit{perfect Langford sequences}. Notice that, a Langford sequence of order $m$ and defect $d=1$ is a Skolem sequence of order $m$.
Bermond, Brower and Germa on one side \cite{BerBroGer}, and Simpson on the other side \cite{Sim} characterized the existence of Langford sequences for every order $m$ and defect $d$.

\begin{theorem}\cite{BerBroGer,Sim}\label{Theo: existence of Lang}
A Langford sequence of order $m$ and defect $d$ exists if and only if the following conditions hold:
(i) $m\ge 2d-1$, and
(ii) $m\equiv 0$ or $1$ (mod $4$) if $d$ is odd; $m\equiv 0$ or $3$ (mod $4$) if $d$ is even.
\end{theorem}

For a complete survey on Skolem-type sequences we refer the reader to \cite{FranMen2009}.

\subsection{Distance labelings}
Let $L=(l_1,l_2,\ldots, l_{2m})$ be a Langford sequence of order $m$ and defect $d$. Consider a path $P$ with $V(P)=\{v_i: i=1,2,\ldots,2m\}$ and $E(P)=\{v_iv_{i+1}:\ i=1,2,2m-1\}$. Then, we can identify $L$ with a labeling $f: V(P)\rightarrow [d,d+m-1]$ in such a way that,  (i) for every $k\in [d,d+m-1]$ there exist exactly two vertices $v_i,v_j\in [1,2m]$ with $f(v_i)=f(v_j)=k$, (ii) the distance $d(v_i,v_j)=k$. Motivated by this fact, we introduce two notions of distance labelings, one of them associated with a positive integer $l$ and the other one associated with a set of positive integers $I$.

Let $G$ be a graph and let $l$ be a positive integer. Consider any function $f:V(G)\rightarrow [0,l]$. We say that $f$ is a {\it distance labeling} of length $l$ (or {\it distance $l$-labeling}) of $G$ if the following two conditions hold, (i) either $f(V(G))=[0,l]$ or $f(V(G))=[1,l]$ and (ii) if there exist two vertices $v_i$, $v_j$ with $f(v_i)=f(v_j)=k$ then $d(v_i,v_j)=k$. Clearly, a graph can have many different distance labelings. We denote by $\lambda (G)$, the {\it labeling length} of $G$, the minimum length $l$ for which a distance $l$-labeling of $G$ exists.
We say that a distance $l$-labeling of $G$ is {\it proper} if for every $k\in [1,l]$ there exist at least two vertices $v_i$, $v_j$ of $G$ with $f(v_i)=f(v_j)=k$. We also say that a proper distance $l$-labeling of $G$ is {\it regular} of degree $r$ (for short {\it $r$-regular}) if for every $k\in [1,l]$ there exist exactly $r$ vertices $v_{i_1}$, $v_{i_2}$, \ldots, $v_{i_r}$ with $f(v_{i_1})=f(v_{i_2})=\ldots=f(v_{i_r})=k$. Clearly, if a graph $G$ admits a proper distance $l$-labeling then $l\le D(G)$, where $D(G)$ is the diameter of $G$.

Let $G$ be a graph and let $J$ be a set of nonnegative integers. Consider any function $f:V(G)\rightarrow J$. We say that $f$ is a {\it distance $J$-labeling} of $G$ if the following two conditions hold, (i) $f(V(G))=J$ and (ii) for any pair of vertices $v_i$, $v_j$ with $f(v_i)=f(v_j)=k$ we have that $d(v_i,v_j)=k$. We say that a distance $J$-labeling is {\it proper} if for every $k\in J\setminus \{0\}$ there exist at least two vertices $v_i$, $v_j$ with $f(v_i)=f(v_j)=k$. We also say that a proper distance $J$-labeling of $G$ is {\it regular} of degree $r$ (for short {\it $r$-regular}) if for every $k\in J\setminus \{0\}$ there exist exactly $r$ vertices $v_{i_1}$, $v_{i_2}$, \ldots, $v_{i_r}$ with $f(v_{i_1})=f(v_{i_2})=\ldots=f(v_{i_r})=k$.
Clearly, a distance $l$-labeling is a distance $J$-labeling in which either $J=[0,l]$ or $J=[1,l]$. Thus, the notion of a $J$-labeling is more general than the notion of a $l$-labeling.

In this paper, we provide the labeling length of some well known families of graphs. We also study the inverse problem, that is, for a given pair of positive integers $l$ and $r$ we ask for the existence of a graph of order $lr$ with a regular $l$-labeling of degree $r$. Finally, we study a similar question when we deal with $J$-labelings. The organization of the paper is the following one. Section 2 is devoted to $l$-labelings, we start calculating the labeling length of complete graphs, paths, cycles and some others families. The inverse problem is studied in the second part of the section. Section 3 is devoted to the inverse problem in $J$-labelings. There are many open problem that remain to be solve, we end the paper by presenting some of them.

\section{Distance $l$-labelings}
We start this section by providing the labeling length of some well known families of graphs.
\begin{proposition}
The complete graph $K_n$ has $\lambda(K_n)=1$.
\end{proposition}

\begin{proof}
By assigning the label $1$ to all vertices of $K_n$, we obtain a  {\it distance $1$-labeling} of it. \qed
\end{proof}
\begin{proposition}
The path $P_n$ has $\lambda (P_n)=\lfloor n/2\rfloor$.
\end{proposition}

\begin{proof}
By previous comment, we know that a Skolem sequence of order $m$ exists if $m\equiv 0$ or $1$ (mod $4$). This fact together with (\ref{eq: extended Skolem}) guarantees the existence of a proper distance $\lfloor n/2\rfloor$-labeling when $n\not\equiv 4, 6$ (mod $8$). By removing one of the end labels of (\ref{eq: extended Skolem}), we obtain a (non proper) distance labeling of length $\lfloor n/2\rfloor$. Thus, we have that $\lambda (P_n)\le \lfloor n/2\rfloor$. Since, there are not three vertices in the path which are at the same distance, this lower bound turns out to be an equality.  \qed
\end{proof}

The sequence that appears in (\ref{eq: extended Skolem}) also works for constructing proper distance labelings of cycles. Thus, we obtain the next result.

\begin{proposition}
Let $n$ be a positive integer. The cycle $C_{n}$ has $\lambda (C_n)=\lfloor n/2\rfloor$.
\end{proposition}

\begin{proof}
 Since, except for $C_3$ there are not three vertices in the cycle which are at the same distance, we have that $\lambda (C_n)\ge \lfloor n/2\rfloor$. The sequence that appears in (\ref{eq: extended Skolem}) allows us to construct a (proper) distance $\lfloor n/2\rfloor$-labeling of $C_n$ when $n$ is odd. Moreover, if $n$ is even we can obtain a distance $\lfloor n/2\rfloor$-labeling of $C_n$ from the sequence that appears in (\ref{eq: extended Skolem}) just by removing the end odd label. \qed
\end{proof}
\begin{proposition}
The star $K_{1,k}$ has $\lambda(K_{1,k})=2$, when $k\ge 3$, and $\lambda(K_{1,k})=1$ otherwise.
\end{proposition}

\begin{proof}
For $k\ge 3$, consider a labeling $f$ that assigns the label $1$ to the central vertex and to one of its leaves, and that assigns label $2$ to the other vertices. Then $f$ is a (proper) distance $2$-labeling of $K_{1,k}$. For $1\le k\le 2$, the sequences $1 -1$ and $0 - 1-1$, where $0$ is assigned to a leaf, give a (proper) distance $1$-labeling of $K_{1,1}$ and $K_{1,2}$, respectively. \qed
\end{proof}

\begin{proposition}
Let $m$ and $n$ be integers with $2\le m\le n$. Then, $\lambda(K_{m,n})=m$. In particular, the graph $K_{m,n}$ admits a proper distance $l$-labeling if and only if, $1\le m\le 2$.
\end{proposition}

\begin{proof}
Let $X$ and $Y$ be the stable sets of $K_{m,n}$, with $|X|\le |Y|$. We have that $D(K_{m,n})=2$, however the maximum number of vertices that are mutually at distance $2$ is $n$. Thus, by assigning label 2 to all vertices, except one, in $Y$, $1$ to the remaining vertex in $Y$ and to one vertex in $X$, $0$ to another vertex of $X$ we still have left  $m-2$ vertices in $X$ to label.\qed
\end{proof}

\begin{proposition}\label{propo: el paràmetre de S_k^n}
Let $n$ and $k$ be positive integers with $k\ge 4$. Let $S_k^n$ be the graph obtained from $K_{1,k}$ by replacing each edge with a path of $n+1$ vertices. Then
$$\lambda (S_k^n)=\left\{\begin{array}{ll}
                           2(n-1), & \hbox{if}\ k= n-1, \\
                           2n-1, & \hbox{if}\ k=n,  \\
                           2n, & \hbox{if}\ k >n.
                         \end{array}
\right.$$
\end{proposition}

\begin{proof} Suppose that $S_k^n$ admits a distance $l$-labeling with $l<2n$. Then, all the labels assigned to leaves should be different. Moreover, although each even label could appear $k$-times, one for each of the $k$ paths that are joined to the star $K_{1,k}$, odd labels appear at most twice (either in the same or in two of the original forming paths). Thus, at least $2n-2$ labels are needed for obtaining a distance labeling of $S_k^n$.
The following construction provides a distance $2(n-1)$-labeling of $S_k^n$, when $k= n-1$. Suppose that we label the central vertices of each path using the pattern $2-4-\ldots -2(n-1)$. Then, add odd labels to the leaves and if there is not enough leaves, replace some of the even labels by the remaining odd labels, in such a way that every even label appears at least once.
In case $k=n$, we need to introduce a new odd label, which corresponds to $2n-1$. Finally, when $k>n$, we cannot complete a distance $l$-labeling without using $2n$ labels. Fig. \ref{Fig_1c} provides a proper $2n$-labeling that can be generalized in that case.\qed
\end{proof}

Fig. \ref{Fig_1a} and \ref{Fig_1b} show proper distance labelings of $S_4^5$ and $S_5^5$, respectively, that have been obtained by using the above constructions, and then, combining pairs of paths (whose end odd labels sum up to $8$) for obtaining a proper distance $8$-labeling and $9$-labeling, respectively.

\begin{figure}[h]
  \centering
  \includegraphics[width=300pt]{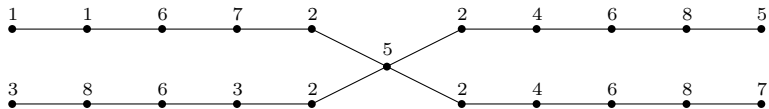}\\
  \caption{A proper distance $8$-labeling of $S_4^5$.}\label{Fig_1a}
\end{figure}

\begin{figure}[h]
  \centering
  \includegraphics[width=300pt]{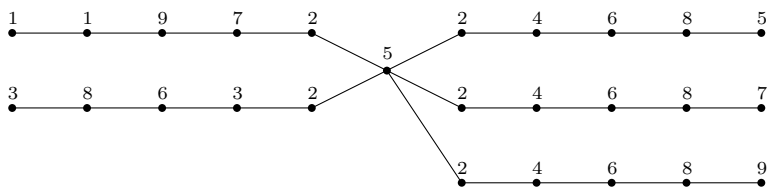}\\
  \caption{A proper distance $9$-labeling of $S_5^5$.}\label{Fig_1b}
\end{figure}
\begin{figure}[h]
  \centering
  \includegraphics[width=300pt]{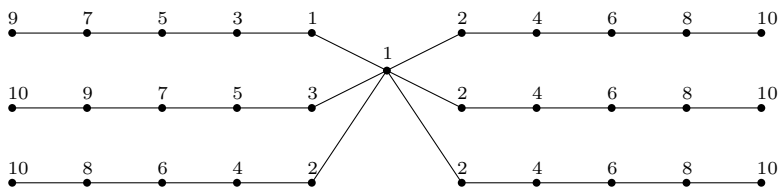}\\
  \caption{A proper distance $10$-labeling of $S_6^5$.}\label{Fig_1c}
\end{figure}

The case $k<n-1$ in Proposition \ref{propo: el paràmetre de S_k^n} requires a more detailed study. Consider the labeling of $S_k^n$ obtained by assigning the labels in the sequence $0-2-4-\ldots -2(n-o)-s^i_1-s^i_2-s^i_o$ to the vertices of the path $P^i$, $i=1,\ldots,k$, where $0$ is the label assigned to the central vertex of $S_k^n$, and $\{s^i_j\}_{i=1,\ldots, k}^{j=1,\ldots, o}$ is the (multi)set of odd labels. By considering the patern $1-1$, $3-1-1-3$, $5-3-1-1-3-5$ to the vertices of one of the paths, it can be checked that, the graph $S_k^n$ admits an $l$-distance labeling with $l\in \{2(n-o), 2(n-o)+1\}$ and $$\lfloor \frac{2n-1}{2k+1}\rfloor\le o\le \lfloor \frac{2n+2}{2k+1}\rfloor.$$ Thus, according to the proof of Proposition \ref{propo: el paràmetre de S_k^n}, we strongly suspect that $$2(n-o)\le \lambda ( S_k^n)\le 2(n-o)+1,$$ where $\lfloor (2n-1)/(2k+1)\rfloor\le o\le \lfloor (2n+2)/(2k+1)\rfloor$.

\begin{proposition}\label{lm: wheel}
For $n\ge 3$, let $W_n$ be the wheel of order $n+1$. Then
$\lambda (W_n)=\lceil n/2\rceil.$
\end{proposition}
\proof Except for $W_3$, all wheels have $D(W_n)=2$. The maximum number of vertices that are mutually at distance $2$ is $\lfloor n/2\rfloor$ and all of them are in the cycle. Thus, by assigning label $2$ to all these vertices, $0$ to one vertex of the cycle and $1$ to the central vertex and to one vertex of the cycle, we still have to label $\lceil n/2\rceil -2$ vertices. \qed .

\begin{proposition}\label{lm: fan}
For $n\ge 2$, let $F_n$ be the fan of order $n+1$. Then
$\lambda (F_n)=\lfloor n/2\rfloor.$
\end{proposition}
\proof Except for $F_2$, all fans have $D(F_n)=2$. The maximum number of vertices that are mutually at distance $2$ is $\lceil n/2\rceil$ and all of them are in the path. Thus, by assigning label $2$ to all these vertices, $0$ to one vertex of the path, $1$ to the central vertex and to one vertex of the path when $n$ is even and to two vertices when $n$ is odd, we still have to label $\lceil n/2\lceil -2$ vertices. \qed
\subsection{The inverse problem}

For every positive integer $l$, there exists a graph $G$ of order $l$ with a trivial $l$-labeling that assigns a different label in $[1,l]$ to each vertex. In this section, we are interested on the existence of a graph $G$ that admits a proper distance $l$-labeling.

We are now ready to state and prove the next result.
\begin{theorem}\label{theo: existence of a graph G with r-regular l-labeling}
For every pair of positive integers $l$ and $r$, there exists a graph $G$ of order $lr$ with a regular $l$-labeling of degree $r$.
\end{theorem}
\begin{proof}
We give a constructive proof. Assume first that $l$ is odd. Let $G$ be the graph obtained from the complete graph $K_r$ by identifying $r-1$ vertices of $K_r$ with one of the end vertices of a path of length $\lfloor l/2\rfloor$ and the remaining vertex of $K_r$ with the central vertex of the graph $S_{r+1}^{\lfloor l/2\rfloor}$. That is, $G$ is obtained from $K_r$ by attaching $2r$ paths of length $\lfloor l/2\rfloor$ to its vertices, $r+1$ to a particular vertex $v_1$ of $K_r$ and exactly one path to each of the remaining vertices $F=\{v_2,v_3,\ldots, v_{r}\}$ of $K_r$. Now, consider the labeling $f$ of $G$ that assigns $1$ to the vertices of $K_r$, the sequence $1-3-\ldots -l$ to the vertices of the paths attached to $F$ and one of the paths attached to $v_1$, and the sequence $1-2-4-\ldots -(l-1)$ to the remaining paths. Then $f$ is a regular $l$-labeling of degree $r$ of $G$. Assume now that $l$ is even. Let $G$ be the graph obtained in the above construction for $l-1$. Then, by adding a leave to each vertex of $G$ labeled with $l-2$ we obtain a new graph $G'$ that admits a regular $l$-labeling $f'$ of degree $r$. The labeling $f'$ can be obtained from the labeling $f$ of $G$, defined above, just by assigning the label $l$ to the new vertices. \qed
\end{proof}

\begin{figure}[h]
  \centering
  \includegraphics[width=164pt]{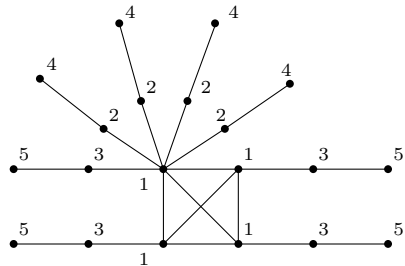}\\
  \caption{A regular $5$-labeling of degree $4$ of a graph $G$.}\label{Fig_2}
\end{figure}

\begin{figure}[h]
  \centering
  \includegraphics[width=176pt]{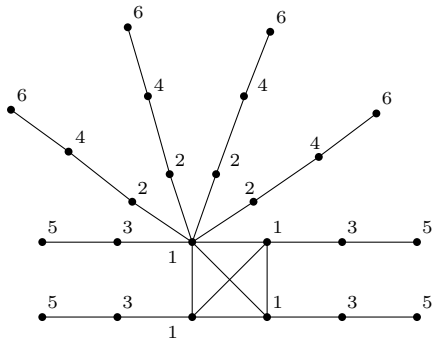}\\
  \caption{A regular $6$-labeling of degree $4$ of a graph $G'$.}\label{Fig_3}
\end{figure}

Notice that, the graph provided in the proof of Theorem \ref{theo: existence of a graph G with r-regular l-labeling} also has $\lambda (G)=l$.
Fig. \ref{Fig_2} and \ref{Fig_3} show examples for the above construction. The pattern provided in the proof of the above theorem, for $r=2$, can be modified in order to obtain the following lower bound for the size of a graph $G$ as in Theorem \ref{theo: existence of a graph G with r-regular l-labeling}.

\begin{proposition}\label{lemma: regular distance $l$-labeling of degree $2$}
For every positive integer $l$ there exists a graph of order $2l$ and size $(l+2)(l+1)/2-2$ that admits a regular distance $l$-labeling of degree $2$.
\end{proposition}
\begin{proof}
Let $G$ be the graph of order $2l$ and size $(l+1)l/2+l-1$, obtained from $K_{l+1}$ and the path $P_l$ by identifying one of the end vertices $u$ of $P_l$ with a vertex $v$ of $K_{l+1}$. Let $f$ be the labeling of $G$ that assigns the sequence $1-2-3-\ldots -l$ to the vertices of $P_l$ and $1-1-2-\ldots-l$ to the vertices of $K_{l+1}$ in such a way that $f(u\equiv v)=1$. Then, $f$ is a $2$-regular $l$-labeling of $G$. \qed
\end{proof}
Thus, a natural question appears.
\begin{question}
Can we find graphs that admit a regular distance $l$-labeling of degree $2$ which are more dense that the one of Proposition \ref{lemma: regular distance $l$-labeling of degree $2$}?
\end{question}
\section{Distance $J$-labelings}

It is clear from definition that to say that a graph admits a (proper) distance $l$-labeling is the same as to say that the graph admits a (proper) distance $[0,l]$-labeling. That is, we relax the condition on the labels, the set of labels is not necessarily a set of consecutive integers. In this section, we study which kind of sets $J$ can appear as a set labels of a graph that admits a distance $J$-labeling.

The following easy fact is obtained from the definition.
\begin{lemma}
Let $G$ be a graph with a proper distance $J$-labeling $f$. Then $J\subset [0,D(G)]$, where $D(G)$ is the diameter of $G$.
\end{lemma}

\subsection{The inverse problem: distance $J$-labelings obtained from sequences.}
We start with a definition. Let $S=(s_1,s_1,\ldots,s_1, s_2,\ldots, s_2, \ldots, s_l,\ldots, s_l)$ be a sequence of nonnegative integers where, (i) $s_i<s_j$ whenever $i<j$ and (ii) each number $s_i$ appears $k_i$ times, for $i=1,2,\ldots, l$. We say that $S$ is a {\it $\delta$-sequence} if there is a simple graph $G$ that admits a partition of the vertices $V(G)=\cup_{i=1}^lV_i$ such that, for all $i\in\{1,2,\ldots, l\}$, $|V_i|=k_i$, and if $u,v\in V_i$ then $d_G(u,v)=s_i$.
The graph $G$ is said to {\it realize} the sequence $S$.

Let $\Sigma=\{s_1<s_2<\ldots <s_l\}$ be a set of nonnegative integers. We say that $\Sigma$ is a $\delta$-set with $n$ degrees of freedom or a {\it $\delta_n$-set} if there is a $\delta$-sequence $S$ of the form $S=(s_1,s_1,\ldots,s_1, s_2,\ldots, s_2, \ldots, s_l,\ldots, s_l)$, in which the following conditions hold: (i) all except $n$ numbers different from zero appear at least twice, and (ii) if $s_1=0$ then $0$ appears exactly once in $S$. We say that any graph realizing $S$ also {\it realizes $\Sigma$}. If $n=0$ we simply say that $\Sigma$ is a {\it $\delta$-set}. Let us notice that an equivalent definition for a $\delta$-set is the following one. We say that $\Sigma$ is a {\it $\delta$-set} if there exists a graph $G$ that admits a proper distance $\Sigma$-labeling.

\begin{proposition}\label{propo: construc_a graph that realizes Sigma}
Let $\Sigma=\{1=s_1<s_2<\ldots <s_l\}$ be a set such that $s_i-s_{i-1}\le 2$, for $i=1,2,\ldots, l$. Then $\Sigma$ is a {\it $\delta$-set}. Furthermore, there is a caterpillar of order $2l$ that realizes $\Sigma$.
\end{proposition}

\begin{proof}
We claim that for each set $\Sigma=\{1=s_1<s_2<\ldots <s_l\}$ such that $s_i-s_{i-1}\le 2$ there is a caterpillar of order $2l$ that admits a $2$-regular distance $\Sigma$-labeling in which the label $s_l$ is assigned to exactly two leaves. The proof is by induction. For $l=1$, the path $P_2$ admits a $2$-regular distance $\{1\}$-labeling, and for $l=2$, the star $K_{1,3}$ and the path $P_4$ admit a $2$-regular distance $\{1,2\}$-labeling and a $2$-regular distance $\{1,3\}$-labeling, respectively. Assume that the claim is true for $l$ and let $\Sigma=\{1=s_1<s_2<\ldots <s_{l+1}\}$ such that $s_i-s_{i-1}\le 2$. Let $\Sigma'=\Sigma\setminus \{s_{l+1}\}$. By the induction hypothesis, there is a caterpillar $G'$ of order $2l$ that admits a regular distance $\Sigma$-labeling of degree $2$ in which the label $s_l$ is assigned to leaves, namely, $u_1$ and $u_2$. Let $u\in V(G')$ be the (unique) vertex in $G'$ adjacent to $u_1$. Then, if $s_{l+1}-s_l=2$, the caterpillar obtained from $G'$ by adding two new vertices $v_1$ and $v_2$ and the edges $u_iv_i$, for $i=1,2$, admits a regular distance $\Sigma$-labeling of degree $2$ in which the label $s_{l+1}$ is assigned to leaves $\{v_1,v_2\}$. Otherwise, if $s_{l+1}-s_l=1$ then the caterpillar obtained from $G'$ by adding two new vertices $v_1$ and $v_2$ and the edges $uv_1$ and $u_2v_2$ admits a regular distance $\Sigma$-labeling of degree $2$ in which the label $s_{l+1}$ is assigned to leaves $\{v_1,v_2\}$.
This proves the claim. To complete the proof, we only have to consider the vertex partition of $G$ defined by the vertices that receive the same label.\qed
\end{proof}

Proposition \ref{propo: construc_a graph that realizes Sigma} provides us with a family of $\delta$-sets, in which, if we order the elements of each $\delta$-set, we get that the differences between consecutive elements are upper bounded by $2$. This fact may lead us to get the idea that the differences between consecutive elements in $\delta$-sets cannot be too large. This is not true in general and we show it in the next result.

\begin{theorem}
Let $\{k_1<k_2<\ldots <k_n\}$ be a set of positive integers. Then there exists a $\delta$-set $\Sigma=\{s_1<s_2<\ldots <s_l\}$ and a set of indices $\{1\le j_1<j_2<\ldots <j_{n}\}$, with $j_{n}<l-1$, such that
$$s_{j_1+1}-s_{j_1}=k_1,\ s_{j_2+1}-s_{j_2}=k_2,\ldots, \ s_{j_n+1}-s_{j_n}=k_n.$$
Moreover, $s_1$ can be chosen to be any positive integer.
\end{theorem}

\begin{proof}
Choose any number $d_1\in \mathbb{N}$ and choose any Langford sequence of defect $d_1$. We let $d_1=s_1$. (Notice that if $d_1=1$ then the sequence is actually a Skolem sequence). Let this Langford sequence be $L_1$. Next, choose a Langford sequence $L_2$ with defect $\max L_1+k_1$.  Next, choose a Langford sequence $L_3$ with defect $\max L_2+k_2$. Keep this procedure until we have used all the values $k_1,k_2,\ldots,  k_n$. At this point create a new sequence $L$, where $L$ is the concatenation of $L_1, L_2,\ldots, L_{n+1}$ and label the vertices of the path $P_r$, $r=\sum_{i=1}^{n+1}|L_i|$, with the elements of $L$ keeping the order in the labeling induced by the sequence $L$. This shows the result.\qed
\end{proof}

The next result shows that there are sets that are not $\delta$-sets.

\begin{proposition}\label{lema: set {2,3}}
The set $\Sigma=\{2,3\}$ is not a $\delta$-set.
\end{proposition}

\begin{proof}
The proof is by contradiction. Assume to the contrary that $\Sigma=\{2,3\}$ is a $\delta$-set. That is to say, we assume that there exists a sequence $S$ consisting of $k_1$ copies of $2$ and $k_2$ copies of $3$ that is a $\delta$-sequence. Let $G$ be a graph that realizes $S$ and $V_1\cup V_2$ the partition of $V(G)$ defined as follows: if $u,v\in V_i$ then $d_G(u,v)=i+1$, for $i=1,2$. It is clear that $V_1$  must be formed by the leaves of a star with center some vertex $a\in V$. Since $a$ is at distance $1$ of any vertex in $V_1$, it follows that $a$ must be in $V_2$ and furthermore, all vertices adjacent to $a$ must be in $V_1$. Thus, there are no two adjacent vertices in the neighborhood of $a$. At this point, let $b\in V_2\setminus \{a\}$. Then, there is a path of the form $a,u_1,u_2,b$, where $u_1\in V_1$ and hence, $u_2,b\in V_2$. This contradicts the fact that $d_G(u_2,b)=1$. \qed
\end{proof}

The above proof works for any set of the form $\Sigma=\{2,n\}$, for $n\ge 3$. Thus, in fact, Proposition \ref{lema: set {2,3}} can be generalized as follows.

\begin{proposition}\label{lema: set {2,n}}
The set $\Sigma=\{2,n\}$ is not a $\delta$-set.
\end{proposition}

Notice that, although $\Sigma=\{2,n\}$ is not a $\delta$-set, it is a $\delta_1$-set, since we can consider a star in which the center is labeled with $n$ and the leaves with $2$.

The next result gives a lower bound on the size of $\delta$-sets in terms of the maximum of the set.

\begin{theorem}\label{theo: the lower boun on the size of Sigma}
Let $\Sigma$ be a $\delta$-set with $s=\max \Sigma$. Then, $|\Sigma|\ge \lceil (s+1)/2\rceil$.
\end{theorem}
\begin{proof}
Let $\Sigma$ be a $\delta$-set with $s=\max \Sigma$. Let $G$ be a graph that realizes $\Sigma$ and let $V(G)=\cup_{i\in \Sigma}V_i$ be the partition defined as follows: if $u,v\in V_i$ then $d_G(u,v)=i$. Let $a_1,a_2\in V_s$. At this point, let $P=a_1=b_1b_2\ldots b_{s+1}=a_2$ be a path of length $s$ starting at $a_1$ and ending at $a_2$. We claim that there are no three vertices in $V(P)$ belonging to the same set $V_j$, $j\in \Sigma$. We proceed by contradiction. Assume to the contrary that there exist vertices $u,v$ and $w\in V(P)$ such that $u,v,w\in V_j$. That is, $d_G(u,v)=d_G(u,w)=d_G(v,w)=j$. However, it is clear that the above cannot happen if we take the distances in the path. That is to say, the following is impossible: $d_P(u,v)=d_P(v,w)=d_P(v,w)=j$. Without loss of generality assume that $d_p(u,v)=k\ne j$. Clearly, $k>j$. Otherwise, we have that $d_G(u,v)\le k$ instead of $d_G(u,v)=j$. Let $P'$ be a path in $G$ of length $j$ that joins $u$ and $v$. Then, the subgraph of $G$ obtained from $P$ by substituting the subpath of $P$ joining $u$ and $v$ by $P'$ contains a subpath of length strictly smaller than $s$. Thus, we obtain a contradiction.
Hence, each set in the partition of $V(G)$ can contain at most two vertices of $P$. Since $|V(P)|=s+1$, it follows that we need at least $\lceil (s+1)/2\rceil$ sets in the partition of $V(G)$. Therefore, we obtain that $|\Sigma|\ge \lceil (s+1)/2\rceil$. \qed
\end{proof}

It is clear that the above proof cannot be improved in general, since from Proposition \ref{propo: construc_a graph that realizes Sigma} we get that the any set of the form $\{1,3,\ldots, 2n+1\}$ is a $\delta$-set and $|\{1,3,\ldots, 2n+1\}|=\lceil (2n+2)/2\rceil$. Furthermore, Proposition \ref{lema: set {2,n}} is an immediate consequence of the above result. It is also worth to mention that there are sets which meet the bound provided in Theorem \ref{theo: the lower boun on the size of Sigma}, however they are not $\delta$-sets. For instance, the set $\{2,3\}$ considered in Lemma \ref{lema: set {2,3}}. From this fact, we see that we cannot characterize $\delta$-sets from, only, a density point of view. Next we want to propose the following open problem.

\begin{problem}\label{problem: delta_0}
Characterize $\delta$-sets.
\end{problem}

Let $\Sigma$ be a set. By construction, a path of order $|\Sigma|$ in which each vertex receives a different labeling of $\Sigma$ defines a distance $|\Sigma|$-labeling. That is, every set is a $\delta_{|\Sigma|}$-set. So, according to that, we propose the next problem.

\begin{problem}\label{problem: delta_n}
Given a set $\Sigma$ is there any construction that provides the minimum $r$ such that $\Sigma$ is a $\delta_r$-set.
\end{problem}

Thus, the above problem is a bit more general than Problem \ref{problem: delta_0}.

\noindent {\bf Acknowledgements}
The research conducted in this document by the first author has been supported by the Spanish Research Council under project
MTM2011-28800-C02-01 and symbolically by the Catalan Research Council
under grant 2014SGR1147.

\end{document}